\documentclass{amsart}%
\usepackage{amsmath,amssymb,bm}%
\usepackage[dvipsnames]{xcolor}%
\usepackage{manyfoot}%
\usepackage{enumitem}%
\usepackage{hyperref}
\usepackage[nameinlink,noabbrev]{cleveref}

\newtheorem*{ineq}{\underline{\normalfont \textsc{Marcinkiewicz--Zygmund's inequality}}}
\newtheorem*{ineq2}{\underline{\normalfont \textsc{Hölder's inequality}}}
\newtheorem*{ineq3}{\underline{\normalfont \textsc{Jensen's inequality}}}
\newtheorem*{ineq4}{\underline{\normalfont \textsc{Lyapunov's inequality}}}
\newtheorem{thm}{Theorem}
\newtheorem{prop}{Proposition}
\newtheorem{cor}{Corollary}
\newtheorem{lem}{Lemma}

\newtheorem{rem}{Remark}

\newtheorem{prob}{Problem}%

\title[On the Submultiplicativity of Matrix Norms Induced by Random Vectors]{On the Submultiplicativity of Matrix Norms Induced by Random Vectors}

\author{Ludovick Bouthat}\email{ludovick.bouthat.1@ulaval.ca}

\address{Département de mathématiques et de statistique, Université Laval, 2325 Rue de l'Université, Québec, G1V 0A6, QC, Canada}

\makeatletter
\newcommand{\opnorm}{\@ifstar\@opnorms\@opnorm}
\newcommand{\@opnorms}[1]{%
  $\left|\mkern-1.5mu\left|\mkern-1.5mu\left|
   #1\right|\mkern-1.5mu\right|\mkern-1.5mu\right|$
}
\newcommand{\@opnorm}[2][]{%
  \mathopen{#1|\mkern-1.5mu#1|\mkern-1.5mu#1|}
  #2
  \mathclose{#1|\mkern-1.5mu#1|\mkern-1.5mu#1|}
}
\makeatother

\begin{document}

\subjclass[2020]{47A30, 15A60, 15A18}

\begin{abstract}
In a recent article, Chávez, Garcia and Hurley introduced a new family of norms $\|\cdot\|_{\mathbf{X},d}$ on the space of $n \times n$ complex matrices which are induced by random vectors $\mathbf{X}$ having finite $d$-moments. Therein, the authors asked under which conditions the norms induced by a scalar multiple of $\mathbf{X}$ are submultiplicative. In this paper, this question is completely answered by proving that this is always the case, as long as the entries of $\mathbf{X}$ have finite $p$-moments for $p=\max\{2+\varepsilon,d\}$.
\end{abstract}


\keywords{Matrix norms, Random vectors, Submultiplicative norms}

\maketitle

\section{Introduction}

Let $d \geq 1$ and $\mathbf{X} = (X_1, X_2,\dots , X_n)$, in which $X_1, X_2,\dots , X_n \in L^d(\Omega,\mathcal{F},\mathbf{P})$ are nondegenerate independent and identically distributed (iid) random variables. In \cite{ChávezGarciaHurley1} and \cite{ChávezGarciaHurley2}, based on their previous work with Konrad Aguilar and Jurij Vol{\v{c}}i{\v{c}} \cite{aguilar2022norms}, Ángel Chávez,
Stephan Ramon Garcia and Jackson Hurley introduced a family of norms on the $n \times n$ complex matrices induced by random vectors. If $\bm{\lambda}$ denotes the vector of real eigenvalues of the Hermitian matrix $A$ and $d\geq 1$, then these norms are defined on the space of $n\times n$ Hermitian matrices $H_n$ by
\begin{equation}\label{eq - def_norm}
    \|A\|_{\mathbf{X},d} \,:=\, \left( \frac{\mathbb{E}\!\left[|\langle \mathbf{X},\bm{\lambda}\rangle |^d\right] }{\Gamma(d+1)}\right)^{\!\frac{1}{d}},
\end{equation}
where $\langle\cdot,\cdot\rangle$ denotes the usual inner product on $\mathbb{R}^n$ and $\Gamma(z)$ is the \emph{gamma function}. That this indeed defines a norm on $H_n$ is shown in \cite[Theorem 1.(a)]{ChávezGarciaHurley1}. Moreover, the authors also showed that $\|\cdot\|_{\mathbf{X},d}$ is a \emph{weakly unitarily invariant} norm on $H_n$, meaning that $\|UAU^*\|_{\mathbf{X},d} = \|A\|_{\mathbf{X},d}$ for any $n\times n$ unitary matrix $U$ and any $n\times n$ Hermitian matrix $A$, which is also \emph{Schur-convex} relative to the vector $\bm{\lambda}(A)$. Recall that a function $f:\mathbb{R}^n\to\mathbb{R}$ is said to be Schur-convex if $f(x)\leq f(y)$ whenever $x$ is \emph{majorized} by $y$.
Additionally, the authors proved that if the entries of $\mathbf{X}$ each have at least $m$ moments, then for fixed $A\in H_n$, the function $f:[1,m]\to \mathbb{R}$ defined by $f(d)=\|A\|_{\mathbf{X},d}$ is continuous. 

Furthermore, the authors demonstrated that if $d\geq 2$ is an even integer, the norm $\|A\|_{\mathbf{X},d}$ defined on $H_n$ can be extended to the entire space $M_n$, and gave an explicit formula for this norm. In particular, they showed that $\|A\|_{\mathbf{X},d}^d$ is trace polynomial, that is, it can be expressed as $\opnorm{Z}_{\mathbf{X},d}^{\smash{d}} = p\big(\operatorname{tr}(Z), \operatorname{tr}(Z^*) \big)$, where $p(x,y)$ is a polynomial in two variables. Moreover, the coefficients of the polynomials are only dependent on the moments of the underlying distribution of the random variables $X_i$. For instance, when $d=2$, a simple computation reveal that for any distribution, the norm $\opnorm{Z}_{\mathbf{X},2}$ is given by
\begin{equation}\label{eq - d=2}
    \opnorm{Z}_{\mathbf{X},2}^2 \,=\, \frac{1}{2} \sigma^2 \|Z\|_{\operatorname{F}}^2 + \frac{1}{2} \mu^2 |\operatorname{tr}(Z)|^2,
\end{equation}
where $\mu$ is the mean and $\sigma$ is the standard deviation of the distribution of the random variables $X_i$.

However, this formulation is complicated to compute in general. For the purposes of this paper, the following formulation, given in \cite[p.\,17]{ChávezGarciaHurley1}, will be used instead:
\begin{equation}\label{eq - def_complete}
    \opnorm{Z}_{\mathbf{X},d} \,:=\, \left( \frac{1}{2\pi \binom{d}{d/2}} \int_0^{2\pi} \big\|e^{it}Z + e^{-it} Z^* \big\|_{\mathbf{X},d}^d \,\mathrm{d}t \right)^{\!\frac{1}{d}},
\end{equation}
where $Z\in M_n$. Observe that this new function is well defined since $e^{it}Z + e^{-it} Z^*$ is always Hermitian for any $t\in [0,2\pi]$. That this function is a norm on $M_n$ which coincide with $\|A\|_{\mathbf{X},d}$ on $H_n$ is proven in \cite[Theorem 1.(e)]{ChávezGarciaHurley1}.

%

Remark that $\opnorm{\cdot}_{\smash{\mathbf{X},d}}$ is not necessarily submultiplicative. Hence, both the factor $(\Gamma(d+1))^{-\frac{1}{d}}$ in \eqref{eq - def_norm} and the factor $\big(2\pi\binom{d}{d/2}\big)^{-1}$ in \eqref{eq - def_complete} are not strictly necessary. These quantities are present only to enhance the elegance of the aforementioned formulations. This naturally raises the question: \emph{Are there better choices of scalars to multiply the norm $\opnorm{\cdot}_{\mathbf{X},d}$}?
%
%
%
%
%
%
%

More specifically, recall that if $\|\cdot\|$ is a norm on $M_n$, then there is a scalar multiple of it (which may depend upon $n$) that is submultiplicative \cite[Theorem 5.7.11]{HornJohnson2012}. Hence, the authors in \cite{ChávezGarciaHurley2} asked the question: \emph{Which random vectors $\mathbf{X}$ induce submultiplicative norms, or norms that become submultiplicative when multiplied by a constant independent of $n$?} This independence condition is important since, if such a constant exists, one can simply scale the distribution $X$ by the appropriate constant to obtain a submultiplicative norm for all $n$, which is not possible if the constant depends on $n$. Of course, when considering matrices $A,B\in H_n$, the product $AB$ might not even be Hermitian. Hence, this question refers to the norm $\opnorm{\cdot}_{\mathbf{X},d}$ and not the norm in \eqref{eq - def_norm}, as the latter is defined only on $H_n$. Consequently, the question of the authors only makes sense for even integers $d\geq 2$.

In this paper, we begin in \Cref{sec - extension} by showing that the function defined in \eqref{eq - def_complete} is a norm for any $d\geq 1$, and not only for even positive integers. Consequently, we extend the question of Chávez, Garcia and Hurley to every real numbers $d\geq 1$. In \Cref{sec - main}, we answer this question by showing that for \emph{any} random variable $X\in L^{\max\{2+\varepsilon,d\}}(\Omega,\mathcal{F},\mathbf{P})$ and \emph{any} real number $d\geq 1$, there exists a positive constant $\gamma_d$, independent of $n$, such that $\gamma_d \opnorm{\cdot}_{\mathbf{X},d}$ is submultiplicative. To achieves this, several estimations of $\opnorm{\cdot}_{\mathbf{X},d}$ are given in \Cref{sec - estimation} , and lower and upper bounds on $\opnorm{\cdot}_{\mathbf{X},d}$ relative to the norm $\opnorm{\cdot}_{\mathbf{X},2}$ are also established in \Cref{sec - bound}. Finally, an enlightening example is presented in \Cref{sec - example}, giving rise to even more questions. 
%



\section{An Extension of a Known Result}
\label{sec - extension}

Initially, the authors of \cite{ChávezGarciaHurley1} only proved that $\|\cdot\|_{\mathbf{X},d}$ is a norm for $d\geq 2$. More recently, the same group of authors showed that $\|\cdot\|_{\mathbf{X},d}$ is also a norm for $1\leq d <2$ \cite{ChávezGarciaHurley2}. However, the authors only established that $\opnorm{\cdot}_{\mathbf{X},d}$ is a norm for even integer $d \geq 2$. Their proof was based on a general result of Aguilar, Chávez, Garcia and Vol{\v{c}}i{\v{c}} \cite[Proposition 15]{aguilar2022norms}, which states the following:
\begin{prop}\normalshape{\cite{aguilar2022norms}}\label{prop - aguilar}
    Let $\mathcal{V}$ be a $\mathbb{C}$-vector space with conjugate-linear involution $*$ and suppose that the real-linear subspace $\mathcal{V}_{\mathbb{R}} = \{v \in \mathcal{V} : v = v^*\}$ of $*$-fixed points has the
    norm $\|\cdot\|$. Then for even $d\geq 2$, the following is a norm on $\mathcal{V}$ that extends $\|\cdot\|$:
    \begin{equation}\label{def - aguilar}
        \opnorm{v}_d \,:=\, \Bigg( \frac{1}{2\pi\binom{d}{d/2}} \int_0^{2\pi}  \|e^{it} v + e^{-it} v^*\|^d \, \mathrm{d}t \Bigg)^{\!\frac{1}{d}}.
    \end{equation}
\end{prop}
Note that $e^{it} v + e^{-it} v^* \in \mathcal{V}_{\mathbb{R}}$ for each $v \in \mathcal{V}$ and $t \in \mathbb{R}$, and the path $t\mapsto \|e^{it} v + e^{-it} v^*\|$ is continuous for each $v \in V$. Consequently, the function $\opnorm{v}_d$ is always well defined. 

In this section, we follow the argument given by the authors in \cite{aguilar2022norms} to show that the $\opnorm{\cdot}_d$ is indeed a norm for every $d\geq 1$. Of course, in \eqref{def - aguilar} and in everything that follows, one need to interpret the binomial coefficient using the gamma function, that is
\[
 \binom{d}{d/2} \,:=\, \frac{\Gamma(d+1)}{\Gamma\hspace{-.5pt}\big(\frac{d}{2}+1\big)^{\!2}} \,=\, \frac{2^d \Gamma\hspace{-.5pt}\big(\frac{d+1}{2}\big)}{\sqrt{\pi} \Gamma\hspace{-.5pt}\big(\frac{d}{2}+1\big)}. 
\]
Note that, in general, the definition of the binomial coefficient for real parameters is $\binom{z}{w} := \lim\limits_{u \rightarrow z}\lim\limits_{v \rightarrow w}\frac{\Gamma(u+1)}{\Gamma(v+1)\Gamma(u-v+1)}$. However, since $d\geq 1$ and since we only concern ourselves with the coefficient $\binom{d}{d/2}$, it is appropriate to forgo the limiting process.


\begin{prop}
    Under the hypothesis of \Cref{prop - aguilar}, $\opnorm{\cdot}_d$ is a norm on $\mathcal{V}$ that extends $\|\cdot\|$ for any $d\geq 1$.
\end{prop}
\begin{proof}
    Clearly, the nonnegativity of $\opnorm{\cdot}_d$ follows from the nonnegativity of $\|\cdot\|$ on $\mathcal{V}_{\mathbb{R}}$. Moreover, the absolute homogeneity of $\opnorm{\cdot}_d$ immediately follows from the $\mathbb{R}$-homogeneity of $\|\cdot\|$ and the periodicity of the integrand in the definition of $\opnorm{\cdot}_d$. Hence, all that remains to be proved is the subadditivity of $\opnorm{\cdot}_d$ and the fact that $\opnorm{\cdot}_d$ extends $\|\cdot\|$. For the subadditivity, simply observe that 
    \begin{align*}
        \opnorm{u+v}_d &= \bigg( \frac{1}{2\pi\binom{d}{d/2}} \int_0^{2\pi}  \|e^{it} (u+v) + e^{-it} (u+v)^*\|^d \, \mathrm{d}t \bigg)^{\!\frac{1}{d}} \\
        &\leq \bigg( \frac{1}{2\pi\binom{d}{d/2}} \int_0^{2\pi}  \!\big(\|e^{it} u + e^{-it} u^*\| + \|e^{it} v + e^{-it} v^*\| \big)^d \, \mathrm{d}t \bigg)^{\!\frac{1}{d}} \\
        &\leq \bigg( \frac{1}{2\pi\binom{d}{d/2}} \int_0^{2\pi}  \!\|e^{it} u + e^{-it} u^*\|^d \, \mathrm{d}t \bigg)^{\!\frac{1}{d}} \!+ \bigg( \frac{1}{2\pi\binom{d}{d/2}} \int_0^{2\pi}  \!\|e^{it} v + e^{-it} v^*\|^d \, \mathrm{d}t \bigg)^{\!\frac{1}{d}} \\
        &= \opnorm{u}_d+\opnorm{v}_d,
    \end{align*}
    where the first inequality holds because of the subadditivity of $\|\cdot\|$ and the monotonicity of the power functions, and the second holds by the triangle inequality for the $L^d$ norms in the space $C[0,2\pi]$. Finally, $\opnorm{\cdot}_d$ extends $\|\cdot\|$ since, if $v\in\mathcal{V}_{\mathbb{R}}$, then $\|e^{it}v+e^{-it}v^*\|=2|\cos(t)| \|v\|$ and
    \[
    \opnorm{v}_d \,=\, \|v\| \bigg( \frac{2^d}{2\pi\binom{d}{d/2}} \int_0^{2\pi}  |\cos(t)|^d \, \mathrm{d}t \bigg)^{\!\frac{1}{d}} \,=\, \|v\|,
    \]
    where the last equality can be found in \cite{bouthat3}. 
\end{proof}

\begin{cor}
    The function $\opnorm{\cdot}_{\mathbf{X},d}$, defined in \eqref{eq - def_complete}, is a norm on $M_n$ that extends $\|\cdot\|_{\mathbf{X},d}$ for any $d\geq 1$.
\end{cor}

Since $\opnorm{\cdot}_{\mathbf{X},d}$ is a norm on $M_n$ for any $d\geq 1$, it makes sense to consider the main question of this paper for any $d\geq 1$. Hence, in the following, we show that for any distribution $X$ and any real number $d\geq 1$, there exists a positive constant $\gamma_d$, independent of $n$, such that $\gamma_d \opnorm{\cdot}_{\mathbf{X},d}$ is submultiplicative. The show this, we first need to establish some inequalities on the norm $\big( \mathbb{E}\hspace{-.3pt}\big[|\langle \mathbf{X},\lambda\rangle |^d\big]\big)^{\frac{1}{d}}$.


\section{Preliminary Estimations}
\label{sec - estimation}


In the following, a meriads of different famous inequality will be used to prove our results. Although most of them are well known classical inequalities, let us explicitly state them for clarity and completeness.

\begin{ineq2}
    Let $X,Y\in (\Omega,\mathcal{F},\mathbf{P})$ be random variables and let $p,q\in [1,\infty]$ be such that $\frac{1}{p}+\frac{1}{q}=1$. Then
    \[
    \mathbb{E}\hspace{-.3pt}\big[ |XY|\big] \,\leq\, \mathbb{E}\hspace{-.3pt}\big[ |X|^p\big]^{\frac{1}{p}} \, \mathbb{E}\hspace{-.3pt}\big[ |Y|^q\big]^{\frac{1}{q}}.
    \]
\end{ineq2}

The following is Jensen's inequality, which comes in many different forms. We present each of them, since they are all used in this paper.

\begin{ineq3}
    Let $(\Omega,\mathcal{F},\mu)$ be a probability space and let $f:\Omega \to \mathbb{R}^d$ be an integrable function and $X$ an integrable real-valued random variable. Moreover, let $x_{1},x_{2},\ldots ,x_{n}\in\mathbb{R}^d$ and $a_{1},a_{2},\ldots ,a_{n}\geq 0$. If $\varphi :\mathbb {R}^d \to \mathbb {R}$ is a convex function, then:
    \begin{enumerate}
        \item ${\displaystyle \varphi \left(\mathbb{E} [X]\right)\leq \mathbb{E} [\varphi (X)]}\quad$ (Probabilistic form);
        \item ${\displaystyle \varphi \!\left(\int _{\Omega }f\,\mathrm {d} \mu \right)\leq \int _{\Omega }\varphi \circ f\,\mathrm{d} \mu }\quad $ (Integral form);
        \item ${\displaystyle \varphi \!\left({\frac {\sum a_{i}x_{i}}{\sum a_{i}}}\right)\leq {\frac {\sum a_{i}\varphi (x_{i})}{\sum a_{i}}}}\quad$ (Finite form).
    \end{enumerate}
    The inequalities are reversed if $\varphi$ is concave.
\end{ineq3}

The following is actually a corollary of both Hölder's and Jensen's inequality. Nonetheless, we present it as its own inequality because of its importance in this paper. Recall that the $k$-th absolute moment of a random variable $X$ is defined as $\mathbb{E}\big[|X|^k\big]$.

\begin{ineq4}
    Let $X\in (\Omega,\mathcal{F},\mathbf{P})$ be a random variable of finite $t$-th absolute moment, and suppose that $0<s\leq t$. Then 
    \[
    \mathbb{E}\hspace{-.3pt}\big[ |X|^s\big]^{\frac{1}{s}} \,\leq\, \mathbb{E}\hspace{-.3pt}\big[ |X|^t\big]^{\frac{1}{t}}.
    \]
\end{ineq4}

\begin{ineq}
{\normalshape\cite{Marcinkiewicz1937}}
    Let $1\leq d < \infty$. If $X_i\in L^d(\Omega,\mathcal{F},\mathbf{P})$, $i=1,\dots,n$, are independent random variables such that $\mathbb{E}[X_i]=0$, then 
    \begin{equation*}
        A_d \mathbb{E} \!\left[ \bigg( \sum_{i=1}^n |X_i|^2 \bigg)^{\!\frac{d}{2}} \right] \,\leq\, \mathbb{E} \!\left[ \bigg| \sum_{i=1}^n X_i \bigg|^{d} \right] \,\leq\, B_d \mathbb{E} \!\left[ \bigg( \sum_{i=1}^n |X_i|^2 \bigg)^{\!\frac{d}{2}} \right],
    \end{equation*}
    where $A_d$ and $B_d$ are positive constants, which depend only on $d$ and not on $n$ nor on the underlying distribution of the random variables. 
\end{ineq}

\begin{rem}\label{rem - constant}
    Many generalizations of the Marcinkiewicz--Zygmund inequality exists. For instance, Zhang showed that the result holds for asymptotically linear negative quadrant dependent (ALNQD) random variables with mean 0 \cite{Zhang2000}. Moreover, Hadjikyriakou \cite{HADJIKYRIAKOU2011678} proved that if $S_n:= X_1+ \cdots +X_n$ are nonnegative $N$-demimartingale, then the upper bound holds with $B_d=\max\{2^{2-d},d-1\}$. In most cases, however, the optimal constants $A_d$ and $B_d$ are not known, although it is clear that $A_d\leq 1 \leq B_d$ (simply consider the case $n=1$). Nonetheless, it is shown in \cite{FERGER201496} that if $d\geq 2$, the constants $A_d=2^{-d}$ and $B_d = 8^{\frac{d}{2}} \Gamma\big(\frac{d+1}{2}\big)/\sqrt{\pi}$ works.
\end{rem}


We now provide a pair of inequalities on $\left( \mathbb{E}\!\left[|\langle \mathbf{X},\bm{\lambda}\rangle |^d\right]\right)^{\!\smash{\frac{1}{d}}}$ of the form $c(X,d) \|A\|_{\operatorname{F}}$, where $\|\cdot\|_{\operatorname{F}}$ denotes the Frobenius norm and $c(X,d)$ is some constant independant of $n$. These estimations allows one to conclude that the norm $\|\cdot\|_{\mathbf{X},d}$ behaves like the Frobenius norm on $H_n$, even as $n\to\infty$, as long as $\mathbb{E}[X]=0$.

\begin{lem}\label{ineq - main}
    Let $d\geq 2$ and let $\mathbf{X} = (X_1, X_2, \dots , X_n)$, in which $X_1, X_2, \dots , X_n \in L^d(\Omega,\mathcal{F},\mathbf{P})$ are iid random variables of distribution $X$ satisfying $\mathbb{E}[X]=0$. If $\bm{\lambda}$ is the vector of eigenvalues of the Hermitian matrix $A$, then
    \begin{equation*}
        A_d \mathbb{E}\hspace{-.3pt}\big[|X|^2\big]^{\frac{d}{2}} \|A\|_{\operatorname{F}}^d \,\leq\, \mathbb{E}\!\left[|\langle \mathbf{X},\bm{\lambda}\rangle |^d \right] \,\leq\, B_d \mathbb{E}\hspace{-.3pt}\big[|X|^d\big] \|A\|_{\operatorname{F}}^d ,
    \end{equation*}
    where $A_d$ and $B_d$ are positive constants which depend only on $d$..
\end{lem}
\begin{proof}
    The Marcinkiewicz--Zygmund inequality implies that
    \[
    A_d \mathbb{E} \!\left[ \bigg( \sum_{i=1}^n |\lambda_i X_i|^2 \bigg)^{\!\frac{d}{2}} \right] \,\leq\, \mathbb{E}\!\left[|\langle \mathbf{X},\bm{\lambda}\rangle |^d \right] \,\leq\, B_d \mathbb{E} \!\left[ \bigg( \sum_{i=1}^n |\lambda_i X_i|^2 \bigg)^{\!\frac{d}{2}} \right].
    \]
    Hence, to establish the desired estimates, it is sufficient to bound the quantity $\mathbb{E} \hspace{-.3pt}\Big[ \big( \sum_{i=1}^n |\lambda_i X_i|^2 \big)^{\!d/2} \Big]$. To obtain the upper bound, observe that the finite form of Jensen's inequality and the convexity of $x^{\smash{d/2}}$ ensures that
    \begin{align*}
        \mathbb{E} \!\left[ \bigg( \sum_{i=1}^n |\lambda_i X_i|^2 \bigg)^{\!\frac{d}{2}} \right] \,&=\, \mathbb{E} \!\left[ \bigg( \frac{\lambda_1^2 |X_1|^2+\cdots+ \lambda_n^2 |X_n|^2}{\lambda_1^2 + \cdots + \lambda_n^2}\bigg)^{\!\frac{d}{2}} \right] \cdot (\lambda_1^2 + \cdots + \lambda_n^2)^{\frac{d}{2}} \\
        &\leq\, \mathbb{E} \!\left[ \frac{ \lambda_1^2 |X_1|^d+\cdots+ \lambda_n^2 |X_n|^d}{\lambda_1^2 + \cdots + \lambda_n^2} \right] \cdot (\lambda_1^2 + \cdots + \lambda_n^2)^{\frac{d}{2}} \\
        %
        %
        &=\, \mathbb{E}\hspace{-.3pt}\big[|X|^d\big]  \frac{ \lambda_1^2 +\cdots+ \lambda_n^2}{\lambda_1^2 + \cdots + \lambda_n^2} \cdot (\lambda_1^2 + \cdots + \lambda_n^2)^{\frac{d}{2}}
        \\
        &=\, \mathbb{E}\hspace{-.3pt}\big[|X|^d\big] \|A\|_{\operatorname{F}}^d.
    \end{align*}
    Hence, $\mathbb{E}\hspace{-.3pt}\big[|\langle \mathbf{X},\bm{\lambda}\rangle |^d \big] \leq B_d \mathbb{E}\hspace{-.3pt}\big[|X|^d\big] \|A\|_{\operatorname{F}}^d$. 
    To establish the lower bound, simply apply the probabilistic form of Jensen's inequality to obtain
    \begin{align*}
        \mathbb{E} \!\left[ \bigg( \sum_{i=1}^n |\lambda_i X_i|^2 \bigg)^{\!\frac{d}{2}} \right] \geq  \bigg( \sum_{i=1}^n \lambda_i^2 \mathbb{E}\hspace{-.3pt}\big[|X_i|^2\big] \bigg)^{\!\frac{d}{2}} \!=  \bigg( \mathbb{E}\hspace{-.3pt}\big[|X|^2\big] \sum_{i=1}^n \lambda_i^2 \bigg)^{\!\frac{d}{2}} \!=\,  \mathbb{E}\hspace{-.3pt}\big[|X|^2\big]^{\frac{d}{2}} \|A\|_{\operatorname{F}}^d ,
    \end{align*}
    which completes the proof.
\end{proof}

\medskip

To conclude this section, we provide another set of inequalities, this time on the norm $\opnorm{\cdot}_{\mathbf{X},d}$ on $M_n$. These play an important role in the following section.

\begin{prop}\label{prop - ineq_easy}
    Let $\mathbf{X} = (X_1, \dots , X_n)$, in which $X_1, \dots , X_n \in L^d(\Omega,\mathcal{F},\mathbf{P})$ are iid random variables. Then
    \begin{enumerate}
        \item $\opnorm{Z}_{\mathbf{X},d} \geq \sqrt{\frac{\pi}{2}} \Big( 2 \Gamma\hspace{-.5pt}\big(\frac{d+1}{2}\big)^{\!2} \Big)^{\!\!-\frac{1}{d}} \opnorm{Z}_{\mathbf{X},2} $ if $d\geq 2$;
        \item $\opnorm{Z}_{\mathbf{X},d} \leq \sqrt{\frac{\pi}{2}} \Big( 2 \Gamma\hspace{-.5pt}\big(\frac{d+1}{2}\big)^{\!2} \Big)^{\!\!-\frac{1}{d}} \opnorm{Z}_{\mathbf{X},2} $ if $1\leq d\leq 2$ and $X_k\in L^2(\Omega,\mathcal{F},\mathbf{P})$.
    \end{enumerate}
\end{prop}
\begin{proof}
    Let us first suppose that $d\geq 2$. In this case, observe that Lyapunov's inequality ensures that $\mathbb{E}\hspace{-.3pt}\big[ |\langle \mathbf{X},\bm{\lambda} \rangle|^d\big] 
        \geq \mathbb{E}\hspace{-.3pt}\big[ |\langle \mathbf{X},\bm{\lambda} \rangle|^2\big]^{\frac{d}{2}}.$ 
    Hence, for any Hermitian matrix $A$ we have
    \begin{align*}
        \|A\|_{\mathbf{X},d}^d \,&=\, \frac{\mathbb{E}\hspace{-.3pt}\big[ |\langle \mathbf{X},\bm{\lambda} \rangle|^d\big]}{\Gamma(d+1)}  \geq\, \frac{\mathbb{E}\hspace{-.3pt}\big[ |\langle \mathbf{X},\bm{\lambda} \rangle|^2\big]^{\!\frac{d}{2}}}{\Gamma(d+1)} \,=\, \frac{2^{\frac{d}{2}}}{\Gamma(d+1)}\|A\|_{\mathbf{X},2}^d.
    \end{align*}
    Therefore, we have
    \begin{align*}
        \opnorm{A}_{\mathbf{X},d}^d \,&=\,  \frac{1}{2\pi \binom{d}{d/2}} \int_0^{2\pi} \big\|e^{it}Z + e^{-it} Z^* \big\|_{\mathbf{X},d}^d \,\mathrm{d}t  \\
        &\geq\, \frac{2^{\frac{d}{2}-1}}{\pi \Gamma(d+1)\binom{d}{d/2}} \int_0^{2\pi} \big\|e^{it}Z + e^{-it} Z^* \big\|_{\mathbf{X},2}^d \,\mathrm{d}t  . \\
        %
    \end{align*}
    It thus follows from the integral version of Jensen's inequality that
    \begin{align*}
        \opnorm{A}_{\mathbf{X},d}^d \,
        %
        &\geq\, \frac{(\pi/2)^{\frac{d}{2}}}{ 2\Gamma\hspace{-.5pt}\big(\frac{d+1}{2}\big)^{\!2} } \cdot \frac{1}{(4\pi)^{\frac{d}{2}}} \int_0^{2\pi} \big\|e^{it}Z + e^{-it} Z^* \big\|_{\mathbf{X},2}^d \,\mathrm{d}t   \\
        &\geq\, \frac{(\pi/2)^{\frac{d}{2}}}{ 2\Gamma\hspace{-.5pt}\big(\frac{d+1}{2}\big)^{\!2} } \bigg( \frac{1}{4\pi} \int_0^{2\pi} \big\|e^{it}Z + e^{-it} Z^* \big\|_{\mathbf{X},2}^2 \,\mathrm{d}t \bigg)^{\!\frac{d}{2}} \\
        &=\, \frac{(\pi/2)^{\frac{d}{2}}}{ 2\Gamma\hspace{-.5pt}\big(\frac{d+1}{2}\big)^{\!2} } \opnorm{A}_{\mathbf{X},2}^d.
    \end{align*}
    Taking the $d$-th root on each sides of the above inequality and simplifying yield the desired result if $d\geq 2$. If $1\leq d \leq 2$, both inequalities are reversed and we are done.
\end{proof}

\section{Bounding the Norm \texorpdfstring{$\left|\mkern-1.5mu\left|\mkern-1.5mu\left|
   \cdot\right|\mkern-1.5mu\right|\mkern-1.5mu\right|_{\mathbf{X},d}$}{||.||}}
   \label{sec - bound}

The proof of our main result, presented in the following section, relies on three important properties. The first of those is presented in this section, and consists in a pair of inequalities between $\opnorm{\cdot}_{\mathbf{X},d}$ and $\opnorm{\cdot}_{\mathbf{X},2}$. The proof of these inequalities needs to be treated in two distinct cases, namely $d\geq 2$ and $1\leq d \leq 2$. To simplify the statement of the results and for better clarity, we present these two cases in distinct propositions.


\smallskip
Since the proof of the case $1\leq d \leq 2$ relies on the case $d\geq 2$, we begin by treating the latter. The following proposition depends on the \emph{$d$-th standardized absolute moment} of the distribution $X$. This quantity is defined as $\Tilde{\mu}_d := \frac{\mathbb{E}[|X-\mu|^d ]}{\sigma^d}$, where $\sigma$ is the standard deviation of the random variables $X_i$ and $\mu$ is the mean of $X$.

\begin{prop}\label{prop - main_ineq}
    Let $d\geq 2$ and let $\mathbf{X} = (X_1, \dots , X_n)$, in which $X_1, \dots , X_n \in L^d(\Omega,\mathcal{F},\mathbf{P})$ are iid random variables of $d$-th standardized absolute moment $\tilde{\mu}_d$. Then
    \[ 
    \sqrt{\frac{\pi}{2}} \Bigg( \frac{1}{2 \Gamma\hspace{-.5pt}\big(\frac{d+1}{2}\big)^{\!2}} \Bigg)^{\!\!\frac{1}{d}} \opnorm{Z}_{\mathbf{X},2}
    \,\leq\, \opnorm{Z}_{\mathbf{X},d} \,\leq\, \sqrt{2} \Bigg(\frac{ \pi B_d \Tilde{\mu}_d }{2 \Gamma\hspace{-.5pt}\big(\frac{d+1}{2}\big)^{\!2}}\Bigg)^{\!\!\frac{1}{d}}  \opnorm{Z}_{\mathbf{X},2} 
    \]
    where $B_d$ is the constant in the the upper bound of the Marcinkiewicz--Zygmund inequality.
\end{prop}
\begin{proof}
    The lower bound is \Cref{prop - ineq_easy}. Hence, let us prove the upper bound. To do so, observe that the finite version of Jensen's inequality ensures that
    \begin{align*}
        |\langle \mathbf{X},\bm{\lambda} \rangle |^d \,=\,  2^d \bigg| \frac{\langle \mathbf{X}-\mu,\bm{\lambda} \rangle + \mu \operatorname{tr}(A)}{2} \bigg|^d \,\leq\, 2^d  \frac{|\langle \mathbf{X}-\mu,\bm{\lambda} \rangle|^d + |\mu \operatorname{tr}(A)|^d}{2}.
    \end{align*}
    Hence, since $\mathbb{E}[X-\mu]=0$, it follows from \Cref{ineq - main} that
    \begin{align*}
        \mathbb{E}\hspace{-.3pt}\big[ |\langle \mathbf{X},\bm{\lambda} \rangle |^d \big] \leq 2^{d-1} \!\left(  \mathbb{E}\hspace{-.3pt}\big[|\langle \mathbf{X}-\mu,\bm{\lambda} \rangle|^d \big] + |\mu \operatorname{tr}(A)|^d \right) \leq 2^{d-1}  \!\left( B_d \mu_d \|A\|_{\operatorname{F}}^d + |\mu \operatorname{tr}(A)|^d \right),
    \end{align*}
    where $\mu_d:= \mathbb{E}\hspace{-.3pt}\big[ |X-\mu|^d\big]$. Now, a second application of Lyapunov's inequality reveal that $\mu_d= \mathbb{E}\hspace{-.3pt}\big[ |X-\mu|^d\big] \geq \mathbb{E}\hspace{-.3pt}\big[ |X-\mu|^2\big]^{\frac{d}{2}} = \sigma^d$. Moreover, according to \Cref{rem - constant}, the constant $B_d$ satisfies $B_d \geq 1$ for any $d\geq 2$. Therefore, $B_d \Tilde{\mu}_d = B_d \frac{\mu_d}{\sigma^d} \geq 1$ and thus
    \newpage
    \begin{align*}
        \mathbb{E}\hspace{-.3pt}\big[ |\langle \mathbf{X},\bm{\lambda} \rangle |^d \big] \,&\leq\,  2^{d-1} \!\left(  B_d \mu_d \|A\|_{\operatorname{F}}^d + |\mu \operatorname{tr}(A)|^d \right) \\
        &=\,  2^{d-1} \!\left( B_d \Tilde{\mu}_d \sigma^d \|A\|_{\operatorname{F}}^d + |\mu \operatorname{tr}(A)|^d \right) \\
        &\leq\, 2^{d-1} \!\left( B_d \Tilde{\mu}_d \sigma^d \|A\|_{\operatorname{F}}^d + B_d\Tilde{\mu}_d|\mu \operatorname{tr}(A)|^d \right) \\
        &=\, 2^{d-1} B_d \Tilde{\mu}_d  \!\left( \sigma^d \|A\|_{\operatorname{F}}^d + |\mu \operatorname{tr}(A)|^d \right).
    \end{align*}
    It then follows from the classical vector $p$-norm inequality that 
    \begin{align}\label{eq - 3}
        \mathbb{E}\hspace{-.3pt}\big[ |\langle \mathbf{X},\bm{\lambda} \rangle |^d \big] \,&\leq\,  2^{d-1} B_d \Tilde{\mu}_d  \Big( \sigma^d \|A\|_{\operatorname{F}}^d + |\mu \operatorname{tr}(A)|^d \Big) \nonumber\\[-2pt]
        &\leq\, 2^{d-1} B_d \Tilde{\mu}_d \Big( \sigma^2 \|A\|_{\operatorname{F}}^2 + \mu^2 |\operatorname{tr}(A)|^2 \Big)^{\frac{d}{2}} \\
        %
        %
        &=\, 2^{\frac{3d}{2}-1} B_d \Tilde{\mu}_d \opnorm{A}_{\mathbf{X},2}^d, \nonumber
    \end{align}
    where the last identity is \eqref{eq - d=2}. 
    Consequently, we find that
    \begin{equation*}
        \|A\|_{\mathbf{X},d}^d \,=\, \frac{\mathbb{E}\hspace{-.3pt}\big[|\langle \mathbf{X},\bm{\lambda} \rangle | ^d \big]}{\Gamma(d+1)} \,\leq\, \frac{2^{\frac{3d}{2}} B_d \Tilde{\mu}_d}{2\Gamma(d+1)}  \opnorm{A}_{\mathbf{X},2}^{d}
    \end{equation*}
    for any Hermitian matrix $A$. In particular, for any matrix $Z\in M_n$ and any $t\in[0,2\pi]$, we have
    \begin{align*}
        \big\|e^{it}Z+e^{-it} Z^* \big\|_{\mathbf{X},d} \,&\leq\, 2\sqrt{2}\bigg(\frac{B_d \Tilde{\mu}_d}{2\Gamma(d+1)}\bigg)^{\!\frac{1}{d}} \opnorm{e^{it}Z+e^{-it} Z^*}_{\mathbf{X},2} \\
        &\leq\, 2\sqrt{2}\bigg(\frac{B_d \Tilde{\mu}_d}{2\Gamma(d+1)}\bigg)^{\!\frac{1}{d}}  \big( \opnorm{e^{it}Z}_{\mathbf{X},2} + \opnorm{e^{-it} Z^*}_{\mathbf{X},2} \big) \\
        &=\, 4\sqrt{2}\bigg(\frac{B_d \Tilde{\mu}_d}{2\Gamma(d+1)}\bigg)^{\!\frac{1}{d}} \opnorm{Z}_{\mathbf{X},2},
    \end{align*}
    since $\opnorm{Z}_{\mathbf{X},2} = \opnorm{e^{it}Z}_{\mathbf{X},2} = \opnorm{e^{-it} Z^*}_{\mathbf{X},2}$. Therefore, a final computation yield
    \begin{align*}
        \opnorm{Z}_{\mathbf{X},d}^d \,&=\, \frac{1}{2\pi \binom{d}{d/2}} \int_0^{2\pi} \big\|e^{it}Z + e^{-it} Z^* \big\|_{\mathbf{X},d}^d \,\mathrm{d}t \\[-2pt]
        &\leq\, \frac{1}{2\pi \binom{d}{d/2}} \int_0^{2\pi} \frac{2^{\frac{5d}{2}} B_d \Tilde{\mu}_d}{2\Gamma(d+1)} \opnorm{Z}_{\mathbf{X},2}^d \,\mathrm{d}t \\
        &=\, \frac{ 2^{\frac{d}{2}} \pi B_d \Tilde{\mu}_d }{2 \Gamma\hspace{-.5pt}\big(\frac{d+1}{2}\big)^{\!2}} \opnorm{Z}_{\mathbf{X},2}^d ,
    \end{align*}
    which is the desired inequality.
\end{proof}


We now address the case $1\leq d \leq 2$. The proof of the following proposition closely follows what was done in \Cref{prop - main_ineq}. However, since $d\leq 2$, several of the arguments using convexity are no longer valid, and more technical inequalities are required.

\begin{prop}\label{prop - 1-2}
    Let $1\leq d\leq 2$ and let $\mathbf{X} = (X_1, \dots , X_n)$, in which $X_1, \dots , X_n \in L^{2+\varepsilon}(\Omega,\mathcal{F},\mathbf{P})$ $(\varepsilon>0)$ are iid random variables of $d$-th standardized absolute moment $\tilde{\mu}_d$. Then
    \[ 
    \frac{2^{-\frac{d}{2}-1} \pi }{ \left(2^{1+\varepsilon} B_{2+\varepsilon} \Tilde{\mu}_{2+\varepsilon}  \right)^{\!\frac{2-d}{\varepsilon}} \Gamma\hspace{-.5pt}\big(\frac{d+1}{2}\big)^{\!2}} \opnorm{Z}_{\mathbf{X},2}^d
    \,\leq\, \opnorm{Z}_{\mathbf{X},d}^d \,\leq\, \left(\frac{\pi}{2}\right)^{\!\frac{d}{2}} \frac{1}{2 \Gamma\hspace{-.5pt}\big(\frac{d+1}{2}\big)^{\!2}}  \opnorm{Z}_{\mathbf{X},2}^d
    \]
    where $B_{2+\varepsilon}$ is the constant in the the upper bound of the Marcinkiewicz--Zygmund inequality.
\end{prop}
\begin{proof}
    The upper bound is \Cref{prop - ineq_easy}. Hence, let us focus on the lower bound. Note that the desired inequality is trivial if $Z=0$. Hence, for the following, suppose without any loss of generality that $Z\neq 0$. Now, observe that Hölder's inequality ensures that
    \begin{align*}
        \mathbb{E}\big[|XY|\big] \,\leq\, \mathbb{E}\hspace{-.3pt}\big[|X|^p\big]^{\frac{1}{p}}\, \mathbb{E}\hspace{-.3pt}\big[|Y|^q\big]^{\frac{1}{q}}.
    \end{align*}
    Setting $X:= |\langle \mathbf{X},\bm{\lambda} \rangle|^{\frac{\varepsilon d}{2+\varepsilon-d}}$, $Y:= |\langle \mathbf{X},\bm{\lambda} \rangle|^{\frac{(2-d)(2+\varepsilon)}{2+\varepsilon-d}}$ and $p:=1+\frac{2-d}{\varepsilon}$, where $\varepsilon>0$, yield
    \begin{align*}
        \mathbb{E}\hspace{-.3pt}\Big[|\langle \mathbf{X},\bm{\lambda} \rangle|^2\Big]^{1+\frac{2-d}{\varepsilon}} \leq\, \mathbb{E}\hspace{-.3pt}\Big[|\langle \mathbf{X},\bm{\lambda} \rangle|^d\Big] \,\mathbb{E}\hspace{-.3pt}\Big[|\langle \mathbf{X},\bm{\lambda} \rangle|^{2+\varepsilon}\Big]^{\frac{2-d}{\varepsilon}} .
    \end{align*}
    Now, recall \eqref{eq - 3} which states that
    \begin{align*}
        \mathbb{E}\hspace{-.3pt}\big[ |\langle \mathbf{X},\bm{\lambda} \rangle |^{2+\varepsilon} \big] \,\leq\, 2^{\frac{4+3\varepsilon}{2}} B_{2+\varepsilon} \Tilde{\mu}_{2+\varepsilon} \opnorm{A}_{\mathbf{X},2}^{2+\varepsilon} \,=\,  2^{1+\varepsilon} B_{2+\varepsilon} \Tilde{\mu}_{2+\varepsilon} \mathbb{E}\hspace{-.3pt}\big[ |\langle \mathbf{X},\bm{\lambda} \rangle |^2 \big]^{1+\frac{\varepsilon}{2}}   .
    \end{align*}
    Hence, we find that
    \begin{align*}
        \mathbb{E}\hspace{-.3pt}\big[|\langle \mathbf{X},\bm{\lambda} \rangle|^2\big]^{1+\frac{2-d}{\varepsilon}} \leq\, \left(2^{1+\varepsilon} B_{2+\varepsilon} \Tilde{\mu}_{2+\varepsilon}  \right)^{\!\frac{2-d}{\varepsilon}}   \mathbb{E}\hspace{-.3pt}\big[|\langle \mathbf{X},\bm{\lambda} \rangle|^d\big] \,\mathbb{E}\hspace{-.3pt}\big[ |\langle \mathbf{X},\bm{\lambda} \rangle |^2 \big]^{\frac{(2-d)(2+\varepsilon)}{2\varepsilon}}
    \end{align*}
    and simplifying yield
    \[
    \mathbb{E}\hspace{-.3pt}\big[|\langle \mathbf{X},\bm{\lambda} \rangle|^2\big]^{\frac{d}{2}} \,\leq\, 
    \left(2^{1+\varepsilon} B_{2+\varepsilon} \Tilde{\mu}_{2+\varepsilon}  \right)^{\!\frac{2-d}{\varepsilon}}   \mathbb{E}\hspace{-.3pt}\big[|\langle \mathbf{X},\bm{\lambda} \rangle|^d\big].
    \]
    Consequently, we obtain that
    \[
    \frac{2^{\frac{d}{2}} }{\left(2^{1+\varepsilon} B_{2+\varepsilon} \Tilde{\mu}_{2+\varepsilon}  \right)^{\!\frac{2-d}{\varepsilon}} \Gamma(d+1)} \|A\|_{\mathbf{X},2}^{d} \,\leq\, \|A\|_{\mathbf{X},d}^d.
    \]
    Hence, we then find that
    \begin{align*}
        \opnorm{Z}_{\mathbf{X},d}^d \,&=\, \frac{1}{2\pi \binom{d}{d/2}} \int_0^{2\pi} \big\|e^{it}Z + e^{-it} Z^* \big\|_{\mathbf{X},d}^d \,\mathrm{d}t \\
        &\geq\, \frac{2^{\frac{d}{2}} }{2\pi \left(2^{1+\varepsilon} B_{2+\varepsilon} \Tilde{\mu}_{2+\varepsilon}  \right)^{\!\frac{2-d}{\varepsilon}} \Gamma(d+1) \binom{d}{d/2}} \int_0^{2\pi} \big\|e^{it}Z + e^{-it} Z^* \big\|_{\mathbf{X},2}^d \,\mathrm{d}t \\
        &=\, \frac{2^{-\frac{3d}{2}-1} }{ \left(2^{1+\varepsilon} B_{2+\varepsilon} \Tilde{\mu}_{2+\varepsilon}  \right)^{\!\frac{2-d}{\varepsilon}} \Gamma\hspace{-.5pt}\big(\frac{d+1}{2}\big)^{\!2}} \int_0^{2\pi} \big\|e^{it}Z + e^{-it} Z^* \big\|_{\mathbf{X},2}^d \,\mathrm{d}t.
    \end{align*}
    In this case, Jensen's inequality cannot be useful. Instead, observe that we have
    \begin{align*}
        \big\|e^{it}Z + e^{-it} Z^* \big\|_{\mathbf{X},2} = \opnorm{e^{it}Z + e^{-it} Z^*}_{\mathbf{X},2} \leq \opnorm{e^{it}Z}_{\mathbf{X},2} + \opnorm{e^{-it} Z^*}_{\mathbf{X},2} = 2\opnorm{Z}_{\mathbf{X},2}.
    \end{align*}
    Therefore, 
    \begin{align*}
        \big\|e^{it}Z + e^{-it} Z^* \big\|_{\mathbf{X},2}^d \,=\, \frac{\big\|e^{it}Z + e^{-it} Z^* \big\|_{\mathbf{X},2}^2}{\big\|e^{it}Z + e^{-it} Z^* \big\|_{\mathbf{X},2}^{2-d}} \,\geq\, \frac{\big\|e^{it}Z + e^{-it} Z^* \big\|_{\mathbf{X},2}^2}{2^{2-d} \opnorm{Z}_{\mathbf{X},2}^{2-d}}
    \end{align*}
    and it follows that
    \begin{align*}
        \opnorm{Z}_{\mathbf{X},d}^d \,&\geq\, \frac{2^{-\frac{3d}{2}-1} }{ \left(2^{1+\varepsilon} B_{2+\varepsilon} \Tilde{\mu}_{2+\varepsilon}  \right)^{\!\frac{2-d}{\varepsilon}} \Gamma\hspace{-.5pt}\big(\frac{d+1}{2}\big)^{\!2}} \int_0^{2\pi} \big\|e^{it}Z + e^{-it} Z^* \big\|_{\mathbf{X},2}^d \,\mathrm{d}t \\
        &\geq\, \frac{2^{-\frac{d}{2}-1} \pi }{ \left(2^{1+\varepsilon} B_{2+\varepsilon} \Tilde{\mu}_{2+\varepsilon}  \right)^{\!\frac{2-d}{\varepsilon}} \Gamma\hspace{-.5pt}\big(\frac{d+1}{2}\big)^{\!2} \opnorm{Z}_{\mathbf{X},2}^{2-d}} \cdot \frac{1}{4\pi} \int_0^{2\pi} \big\|e^{it}Z + e^{-it} Z^* \big\|_{\mathbf{X},2}^2 \,\mathrm{d}t \\
        &=\, \frac{2^{-\frac{d}{2}-1} \pi }{ \left(2^{1+\varepsilon} B_{2+\varepsilon} \Tilde{\mu}_{2+\varepsilon}  \right)^{\!\frac{2-d}{\varepsilon}} \Gamma\hspace{-.5pt}\big(\frac{d+1}{2}\big)^{\!2}} \opnorm{Z}_{\mathbf{X},2}^d,
    \end{align*}
    which is what we wanted to show.
\end{proof}

\section{The Submultiplicativity of \texorpdfstring{$\left|\mkern-1.5mu\left|\mkern-1.5mu\left|
   \cdot\right|\mkern-1.5mu\right|\mkern-1.5mu\right|_{\mathbf{X},d}$}{||.||}}
   \label{sec - main}


Before addressing our main theorem, let us focus on a more elementary case, i.e., $d=2$. This will allow us to better grasp the problem while also obtaining a sharp answer in this situation. Moreover, the answer will be crucial in the proof of the general case.

Now, according to \cite[Theorem 1.(e)]{ChávezGarciaHurley1}, the $d$-th power of the norm $\opnorm{Z}_{\mathbf{X},d}$ is a trace polynomial. Moreover, recall that when $d=2$, we have
\begin{equation*}
    \opnorm{Z}_{\mathbf{X},2}^2 \,=\, \frac{1}{2} \sigma^2 \|Z\|_{\operatorname{F}}^2 + \frac{1}{2} \mu^2 |\operatorname{tr}(Z)|^2.
\end{equation*}
This elegant formulation makes it possible to show that in this special case, $\opnorm{Z}_{\mathbf{X},2}$ is a submultiplicative norm when multiplied by a constant only dependant on the mean $\mu$ and the standard deviation $\sigma$ of the $X_i$.

\begin{prop}\label{prop - d=2}
    Let $d=2$ and let $\mathbf{X} = (X_1,\dots , X_n)$, in which $X_1,\dots , X_n \in L^d(\Omega,\mathcal{F},\mathbf{P})$ are iid random variables. Then there exists a positive $\gamma$, independent of $n$, such that $\gamma\opnorm{Z}_{\mathbf{X},2}$ is a submultiplicative norm on $M_n$ which is weakly unitarily invariant norm on $H_n$.
\end{prop}
\begin{proof}
    Let $\gamma^2:= \frac{2\sigma^2+2\mu^2}{\sigma^4}$ (which is independent from $n$) and observe that 
    \begin{align*}\label{eq - ineq}
        \gamma \opnorm{Z}_{\mathbf{X},2} \,&=\, \frac{\sqrt{\sigma^2+\mu^2}}{\sigma^2} \sqrt{\sigma^2 \|Z\|_{\operatorname{F}}^2 +  \mu^2 |\operatorname{tr}(Z)|^2} \,\geq\, \frac{\sqrt{\sigma^2+\mu^2}}{\sigma^2} \sqrt{\sigma^2 \|Z\|_{\operatorname{F}}^2} \\
        &=\, \frac{\sqrt{\sigma^2+\mu^2}}{\sigma}  \|Z\|_{\operatorname{F}}. \nonumber
    \end{align*}
    It then follows from Cauchy--Schwarz and the submultiplicativity of the Frobenius norm that 
    \begin{align*}
        \gamma \opnorm{AB}_{\mathbf{X},2} \,&=\, \frac{\sqrt{\sigma^2+\mu^2}}{\sigma^2} \sqrt{\sigma^2 \|AB\|_{\operatorname{F}}^2 +  \mu^2 |\operatorname{tr}(AB)|^2} \\
        &\leq\, \frac{\sqrt{\sigma^2+\mu^2}}{\sigma^2} \sqrt{\sigma^2 \|A\|_{\operatorname{F}}^2 \|B\|_{\operatorname{F}}^2 +  \mu^2 \|A\|_{\operatorname{F}}^2 \|B\|_{\operatorname{F}}^2} \\
        &=\, \frac{\sqrt{\sigma^2+\mu^2}}{\sigma^2} \sqrt{\sigma^2 +  \mu^2} \|A\|_{\operatorname{F}} \|B\|_{\operatorname{F}} \\
        &=\, \bigg(\frac{\sqrt{\sigma^2+\mu^2}}{\sigma} \|A\|_{\operatorname{F}} \bigg) \bigg(\frac{\sqrt{\sigma^2+\mu^2}}{\sigma} \|B\|_{\operatorname{F}} \bigg) \\
        &\leq\, \left( \gamma \opnorm{A}_{\mathbf{X},2} \right) \!\left( \gamma \opnorm{B}_{\mathbf{X},2} \right).
    \end{align*}
    Thus, $\gamma \opnorm{\cdot}_{\mathbf{X},2}$ is always a submultiplicative norm.
\end{proof}

\begin{rem}
    The constant $\gamma^2= \frac{2\sigma^2+2\mu^2}{\sigma^4}$ in the proof above is smallest possible, if independence is required from $n$. Indeed, consider the matrix $A_n=J_n-I_n$, where $J_n$ is the all-ones matrix. Then an easy computation yield
    \[
    \gamma \opnorm{A_n}_{\mathbf{X},2} = \frac{\sqrt{\sigma^2+\mu^2}}{\sigma} \sqrt{n^2-n}
    \]
    and
    \[
    \gamma \opnorm{A_n^2}_{\mathbf{X},2} = \frac{\sqrt{\sigma^{2}+\sigma^{2}}}{\sigma^{2}}\sqrt{\sigma^{2}n\left(n-1\right)\left(n^{2}-3n+3\right)+\mu^{2}n^{2}\left(n-1\right)^{2}}.
    \]
    Therefore, it follows that
    \begin{align*}
        \frac{\gamma \opnorm{A_n^2}_{\mathbf{X},2}}{\gamma^2 \opnorm{A_n}_{\mathbf{X},2}^2} \,=\, \sqrt{1-\frac{\sigma^{2}}{\sigma^{2}+\mu^{2}}\frac{2n-3}{n(n-1)}} \,\xrightarrow{n\to\infty}\, 1.
    \end{align*}
    
    Note that, as a direct consequence, the norm $\opnorm{\cdot}_{\mathbf{X},2}$ is submultiplicative for any $n\geq 1$ if and only if $\sigma^2 \geq 1+\sqrt{1+2\mu^2}$.
\end{rem}



After treating the case $d=2$, we now seek to establish the general case: regardless of the distribution of the random variables of the vector $\mathbf{X}$ and the parameter $d\geq 1$, the induced norm $\opnorm{\cdot}_{\mathbf{X},d}$ is always submultiplicative or becomes submultiplicative when multiplied by an appropriate constant independent of $n$. The main tool used in the proof, along with the estimates proved in the previous section, is the following result, which was briefly refered to earlier. We now state the theorem in all generality.

\begin{thm}{\normalshape \cite[Theorem 5.7.11]{HornJohnson2012}}\label{thm - useful}
    Let $N(\cdot)$ be a norm on $M_n$ and let
    \[
    c(N) \,:=\, \max_{N(A)=1=N(B)} N(AB).
    \]
    For $\gamma>0$, $\gamma N(\cdot)$ is a submultiplicative norm on $M_n$ if and only if $\gamma \geq c(N)$. Moreover, if $\|\cdot\|$ is a submultiplicative norm on $M_n$, if $C_m$ and $C_M$ are positive constants such that
    \begin{equation*}
        C_m \|A\| \,\leq\, N(A) \,\leq\, C_M \|A\|\quad \text{for all } A\in M_n
    \end{equation*}
    and if $\gamma_0:= C_M/C_m^2$, then $\gamma_0 N(\cdot)$ is a matrix norm and thus $\gamma_0 \geq c(N)$.
\end{thm}

We now have all the necessary ingredients to establish our main result.

\begin{thm}\label{thm - main}
    Let $d\geq 1$ and let $\mathbf{X} = (X_1, \dots , X_n)$, in which $X_1, \dots , X_n$ are iid random variables. If $X_1 \in L^p(\Omega,\mathcal{F},\mathbf{P})$, where $p=\max\{d,2+\varepsilon\}$ for some $\varepsilon>0$, then there exists a positive $\gamma_d$, independent of $n$, such that $\gamma_d\opnorm{Z}_{\mathbf{X},d}$ is a submultiplicative norm on $M_n$ which is weakly unitarily invariant norm on $H_n$.
\end{thm}
\begin{proof}
    Let $\gamma^2:= \frac{2\sigma^2+2\mu^2}{\sigma^4}$, as defined in \Cref{prop - d=2}, and let us first suppose that $d\geq 2$. In this case, by \Cref{prop - main_ineq}, we have 
    \[
    \sqrt{\frac{\pi}{2}} \Bigg( \frac{1}{2 \Gamma\hspace{-.5pt}\big(\frac{d+1}{2}\big)^{\!2}} \Bigg)^{\!\!\frac{1}{d}} 
 \gamma\opnorm{Z}_{\mathbf{X},2}
    \,\leq\, \gamma\opnorm{Z}_{\mathbf{X},d} \,\leq\, \sqrt{2} \Bigg(\frac{ \pi B_d \Tilde{\mu}_d }{2 \Gamma\hspace{-.5pt}\big(\frac{d+1}{2}\big)^{\!2}}\Bigg)^{\!\!\frac{1}{d}}  \gamma\opnorm{Z}_{\mathbf{X},2} .
    \]
    Hence, since $\gamma\opnorm{\cdot}_{\mathbf{X},2}$ is submultiplicative by \Cref{prop - d=2}, \Cref{thm - useful} ensures that $\gamma_d \opnorm{Z}_{\mathbf{X},d}$ is submultiplicative, where
    \begin{align*}
        \gamma_d \,:=&\, \gamma\cdot \sqrt{2} \gamma\Bigg(\frac{ \pi B_d \Tilde{\mu}_d }{2 \Gamma\hspace{-.5pt}\big(\frac{d+1}{2}\big)^{\!2}}\Bigg)^{\!\!\frac{1}{d}}   \!\Bigg/   \frac{\pi}{2} \gamma^2\hspace{-.5pt}\Bigg(\frac{1}{2 \Gamma\hspace{-.5pt}\big(\frac{d+1}{2}\big)^{\!2}} \Bigg)^{\!\!\frac{2}{d}} =\, 
        \frac{2\sqrt{2}}{\pi } \Bigg( 2\pi B_d \Tilde{\mu}_d \Gamma\!\hspace{-.5pt}\left(\frac{d+1}{2}\right)^{\!\!2}\Bigg)^{\!\!\frac{1}{d}}. 
    \end{align*}

    Let us now suppose that $1\leq d \leq 2$. Since $X_1 \in L^p(\Omega,\mathcal{F},\mathbf{P})$, where $p=\max\{d,2+\varepsilon\}$, \Cref{prop - 1-2} ensures that 
    \[ 
    \left(\frac{2^{-\frac{d}{2}-1} \pi }{ \left(2^{1+\varepsilon} B_{2+\varepsilon} \Tilde{\mu}_{2+\varepsilon}  \right)^{\!\frac{2-d}{\varepsilon}} \Gamma\hspace{-.5pt}\big(\frac{d+1}{2}\big)^{\!2}} \right)^{\!\!\frac{1}{d}} \!\!\gamma\opnorm{Z}_{\mathbf{X},2}
    \,\leq\, \gamma\opnorm{Z}_{\mathbf{X},d} \,\leq\, \sqrt{\frac{\pi}{2}} \Bigg( \frac{1}{2 \Gamma\hspace{-.5pt}\big(\frac{d+1}{2}\big)^{\!2}} \Bigg)^{\!\!\frac{1}{d}} 
 \gamma\opnorm{Z}_{\mathbf{X},2}.
    \]
    Once again, since $\gamma\opnorm{\cdot}_{\mathbf{X},2}$ is submultiplicative, \Cref{thm - useful} finally ensures that $\gamma_d \opnorm{Z}_{\mathbf{X},d}$ is submultiplicative, where
    \begin{align*}
        \gamma_d \,:=&\, \gamma\cdot \sqrt{\frac{\pi}{2}} \gamma \Bigg( \frac{1}{2 \Gamma\hspace{-.5pt}\big(\frac{d+1}{2}\big)^{\!2}} \Bigg)^{\!\!\frac{1}{d}}    \!\Bigg/   \gamma^2\!\left(\frac{2^{-\frac{d}{2}-1} \pi }{ \left(2^{1+\varepsilon} B_{2+\varepsilon} \Tilde{\mu}_{2+\varepsilon}  \right)^{\!\frac{2-d}{\varepsilon}} \Gamma\hspace{-.5pt}\big(\frac{d+1}{2}\big)^{\!2}} \right)^{\!\!\frac{2}{d}}.
    \end{align*}
    Since $\gamma_d$ does not depend on $n$ for any $d\geq 1$, we are done.
\end{proof}

\smallskip

\section{Example: Symmetric \texorpdfstring{$\alpha$}{α}-Stable Distributions}
\label{sec - example}

A distribution is said to be \emph{stable} if a linear combination, involving strictly positive coefficients, of two independent random variables following this distribution retains the same distribution, up to location and scale parameters. Stable distributions are characterized by only four parameters: the stability parameter $\alpha \in (0, 2]$, the skewness parameter $\beta \in [-1, 1]$, the scale parameter $\gamma \in (0, \infty)$, and the location parameter $\delta \in \mathbb{R}$. When $\beta = \delta = 0$, the resulting distribution is called a \emph{symmetric $\alpha$-stable distribution}, denoted by $\mathbf{S}(\alpha, \gamma)$, and its characteristic function can be expressed as 
\[
\exp\big(-\left|\gamma x\right|^{\alpha}\big).
\]
Note that the cases where $\alpha = 2$ represent the normal distribution, while $\alpha = 1$ represents the Cauchy distribution. These are in fact the only instances of a symmetric $\alpha$-stable distribution where there exists a closed form for the probability density function \cite[p.\,2]{Nolan}.

These distributions are especially useful in the context of this article, as they allow to easily compute the distribution of the random variable $Y:= \langle \mathbf{X},\bm{\lambda} \rangle$. Indeed, according to \cite[Proposition 1.4]{Nolan}, if $X$ follows a symmetric $\alpha$-stable distribution with a scale parameter $\gamma$, then 
\vspace{-4pt}
\[
\lambda X \,\sim\, \mathbf{S}(\alpha,|\lambda|\gamma).
\]
Moreover, if $X_1$ and $X_2$ follow symmetric $\alpha$-stable distributions with scale parameters $\gamma_1$ and $\gamma_2$, respectively, their sum $X_1+X_2$ follows the distribution $\mathbf{S}(\alpha,\gamma)$, where $\gamma=(\gamma_1^\alpha+\gamma_2^\alpha)^{1/\alpha}$.

\smallskip

In the following, suppose that $\alpha\in (1,2)$ and let $\gamma_\alpha := 4/\Gamma\big(\frac{\alpha-1}{\alpha}\big)$. Under these conditions, if $X$ follows the distribution $\mathbf{S}(\alpha,\gamma_\alpha)$ and $\lambda_j$ represents the eigenvalues of the Hermitian matrix $A$, then 
\[
Y= \langle \mathbf{X},\bm{\lambda} \rangle \,\sim\, \mathbf{S}(\alpha,\gamma_\alpha \|A\|_{S_\alpha}).
\] 
Furthermore, since $\alpha\in(1,2)$, it is known that $\mathbb{E}[|X|^p] < \infty$ if and only if $p\in (-1,\alpha)$ \cite[p.\,108]{Nolan}. Consequently, $X\not\in L^{2+\varepsilon}(\Omega,\mathcal{F},\mathbf{P})$ for any $\varepsilon>0$. In fact, we even have $X\not\in L^{\alpha}(\Omega,\mathcal{F},\mathbf{P})$ for any $\alpha\in(1,2)$.

Let us now show that with this distribution, $\opnorm{A}_{\mathbf{X},1}$ is submultiplicative even though $X\not\in L^{\smash{2+\varepsilon}}(\Omega,\mathcal{F},\mathbf{P})$, thereby showing that the assumption in \Cref{thm - main} is not necessary. To do so, we first need to compute $\mathbb{E}\!\left[|\langle \mathbf{X},\bm{\lambda}\rangle |\right]$. However, this is not a trivial task since the probability density function of $X$ is not easily expressible. Therefore, in order to perform this calculation, we employ the following formula derived by Lukacs \cite[(2.3.11)]{lukacs1970characteristic}: if $\varphi(x)$ represents the characteristic function of the random variable $X$ and $d$ is an odd integer, then 
\[
\mathbb{E}\!\left[|X |^d\right] \,=\, \frac{1}{2\pi i^{d+1}} \int_{-\infty}^\infty \big[\varphi^{(d)}(t)-\varphi^{(d)}(-t) \big] \,\mathrm{d}t.
\]
If $d=1$ and $X\sim \mathbf{S}(\alpha,\gamma_\alpha)$, then $Y=\langle \mathbf{X},\bm{\lambda}\rangle \sim \mathbf{S}\big(\alpha,\gamma_\alpha \|A\|_{S_\alpha} \big)$ and the characteristic function of $Y$ is 
$
\varphi(x) = \exp\big(-\gamma_\alpha^\alpha \|A\|_{S_\alpha}^\alpha |x|^{\alpha}\big)
$. Therefore, 
\begin{align*}
    \mathbb{E}\!\left[|\langle \mathbf{X},\bm{\lambda}\rangle|\right] \,&=\,  \frac{\alpha \gamma_\alpha^{2}\|A\|_{S_\alpha}^2}{\pi}\int_{-\infty}^{\infty}\left|t\gamma_\alpha \|A\|_{S_\alpha}\right|^{\alpha-2}\exp\left(-\left|t\gamma_\alpha \|A\|_{S_\alpha}\right|^{\alpha}\right) \mathrm{d}t \\
    &=\, \frac{2\gamma_\alpha \|A\|_{S_\alpha}}{\pi}\int_{0}^{\infty}u^{-\frac{1}{\alpha}}\exp\left(-u\right) \mathrm{d}u \,=\, \frac{2\gamma_\alpha}{\pi}\Gamma\!\left(\tfrac{\alpha-1}{\alpha}\right) \|A\|_{S_\alpha} \\
    &=\, \frac{8}{\pi} \|A\|_{S_\alpha}, 
\end{align*}
since $\gamma_\alpha=4/\Gamma\big(\frac{\alpha-1}{\alpha}\big)$. Consequently, by \eqref{eq - def_norm}, $\|A\|_{\mathbf{X},1} = \frac{8}{\pi} \|A\|_{S_\alpha}$. 
Now, observe that on the one hand, we have
\begin{align*}
    \opnorm{Z}_{\mathbf{X},1} \,&=\, \frac{1}{8} \int_0^{2\pi} \big\|e^{it}Z + e^{-it} Z^* \big\|_{\mathbf{X},1} \,\mathrm{d}t \,=\,  \frac{1}{\pi}  \int_0^{2\pi} \big\|e^{it}Z + e^{-it} Z^* \big\|_{S_\alpha} \mathrm{d}t \\
    &\leq\, \frac{1}{\pi} \int_0^{2\pi} \big(\big\|e^{it}Z\big\|_{S_\alpha} + \big\|e^{-it} Z^* \big\|_{S_\alpha}\big) \,\mathrm{d}t \,=\, \frac{1}{\pi} \int_0^{2\pi} 2\|Z\|_{S_\alpha}  \mathrm{d}t \\
    &=\, 4  \|Z\|_{S_\alpha}.
\end{align*}
On the other, the integral form of Jensen's inequality and the convexity of $\|\cdot\|_{S_\alpha}$ together yield 
\begin{align*}
    \opnorm{Z}_{\mathbf{X},1} \,&=\,  \frac{1}{\pi}  \int_0^{2\pi} \big\|e^{it}Z + e^{-it} Z^* \big\|_{S_\alpha} \mathrm{d}t \,=\, \frac{1}{\pi}  \int_0^{2\pi} \big\|Z + e^{-2it} Z^* \big\|_{S_\alpha} \mathrm{d}t \\
    &\geq\, \frac{1}{\pi}   \left\| \int_0^{2\pi} \big(Z + e^{-2it} Z^*\big) \,\mathrm{d}t\right\|_{S_\alpha} =\, \frac{1}{\pi}   \left\| 2\pi Z\right\|_{S_\alpha} =\, 2 \|Z\|_{S_\alpha}. 
\end{align*}
Therefore, 
\[
2\|Z\|_{S_\alpha} \,\leq\, \opnorm{Z}_{\mathbf{X},1} \,\leq\, 4 \|Z\|_{S_\alpha}
\]
and it finally follows from \Cref{thm - useful} that 
\[
\frac{4}{2^2} \opnorm{Z}_{\mathbf{X},1} \,=\, \opnorm{Z}_{\mathbf{X},1}
\]
is a submultiplicative norm.

\section{Concluding remarks}

This family of norms induced by random vectors is a rich subject which gives rise to several questions. For instance, four questions are proposed in \cite{ChávezGarciaHurley1} about these norms. 
%
%
To conclude, let us add three more questions to this list:

\begin{prob}
    The example in \Cref{sec - example} shows that the hypothesis $X\in L^{2+\varepsilon}(\Omega,\mathcal{F},\mathbf{P})$ is not necessary in \Cref{thm - main}. Can we remove the hypothesis entirely?
\end{prob}

\begin{prob}
    Characterize those $\mathbf{X}$ that give rise to norms $\|\cdot\|_{\mathbf{X},d}$ which, under multiplication by an appropriate scalar $\gamma_d$ independent of $n$, remain a norm when $d\to\infty$.
\end{prob}

\begin{prob}
    Generalize this family of norms to infinite-dimensional compact self-adjoint operators.
\end{prob}

\bibliographystyle{plain}
\bibliography{ref}

\begin{thebibliography}{10}

\bibitem{aguilar2022norms}
Konrad Aguilar, {\'A}ngel Ch{\'a}vez, Stephan~Ramon Garcia, and Jurij
  Vol{\v{c}}i{\v{c}}.
\newblock Norms on complex matrices induced by complete homogeneous symmetric
  polynomials.
\newblock {\em Bulletin of the London Mathematical Society}, 54(6):2078--2100,
  2022.

\bibitem{bouthat3}
Ludovick Bouthat, Javad Mashreghi, and Fr\'{e}d\'{e}ric Morneau-Gu\'{e}rin.
\newblock Monotonicity of certain left and right {R}iemann sums.
\newblock In {\em Recent developments in operator theory, mathematical physics
  and complex analysis}, volume 290 of {\em Oper. Theory Adv. Appl.}, pages
  89--113. Birkh\"{a}user/Springer, Cham, 2023.

\bibitem{ChávezGarciaHurley1}
{\'A}ngel Chávez, Stephan~Ramon Garcia, and Jackson Hurley.
\newblock Norms on complex matrices induced by random vectors.
\newblock {\em Canadian Mathematical Bulletin}, 66(3):808--826, 2023.

\bibitem{ChávezGarciaHurley2}
{\'A}ngel Chávez, Stephan~Ramon Garcia, and Jackson Hurley.
\newblock Norms on complex matrices induced by random vectors {II}: Extension
  of weakly unitarily invariant norms.
\newblock {\em Canadian Mathematical Bulletin}, 2023.

\bibitem{FERGER201496}
Dietmar Ferger.
\newblock Optimal constants in the {M}arcinkiewicz--{Z}ygmund inequalities.
\newblock {\em Statistics \& Probability Letters}, 84:96--101, 2014.

\bibitem{HADJIKYRIAKOU2011678}
Milto Hadjikyriakou.
\newblock {M}arcinkiewicz--{Z}ygmund inequality for nonnegative
  ${N}$-demimartingales and related results.
\newblock {\em Statistics \& Probability Letters}, 81(6):678--684, 2011.

\bibitem{HornJohnson2012}
Roger~A. Horn and Charles~R. Johnson.
\newblock {\em Matrix Analysis}.
\newblock Cambridge University Press, Cambridge, 2 edition, 2012.

\bibitem{lukacs1970characteristic}
E.~Lukacs.
\newblock {\em Characteristic Functions}.
\newblock Griffin books of cognate interest. Hafner Publishing Company, London,
  1970.

\bibitem{Marcinkiewicz1937}
Józef Marcinkiewicz and Antoni Zygmund.
\newblock Sur les fonctions indépendantes.
\newblock {\em Fundamenta Mathematicae}, 29(1):60--90, 1937.

\bibitem{Nolan}
John~P. Nolan.
\newblock {\em Univariate stable distributions: models for heavy tailed data}.
\newblock Springer Series in Operations Research and Financial Engineering.
  Springer, Cham, 2020.

\bibitem{Zhang2000}
L.-X. Zhang.
\newblock A {F}unctional {C}entral {L}imit {T}heorem for {A}symptotically
  {N}egatively {D}ependent {R}andom {F}ields.
\newblock {\em Acta Mathematica Hungarica}, 86(3):237--259, 2000.

\end{thebibliography}

\end{document}